\documentclass[12pt,a4paper]{amsart}
\usepackage{graphicx}
\usepackage{latexsym}
\usepackage{amsxtra}
\usepackage{amsmath}
\usepackage{amsfonts}
\usepackage{amssymb}
\usepackage[figuresright]{rotating}

\newtheorem{theorem}{Theorem}[section]
\newtheorem{corollary}[theorem]{Corollary}

\newtheorem{proposition}[theorem]{Proposition}

\title{New Eulerian numbers of type $D$}
\author{Anna Borowiec}
\author{Wojciech M{\l}otkowski}
\thanks{
W.~M. is supported by the Polish
National Science Center grant No. 2012/05/B/ST1/00626.}
\address{Instytut Matematyczny,
Uniwersytet Wroc{\l}awski,
Plac~Grunwaldzki~2/4,
50-384 Wroc{\l}aw, Poland}
\email{anna.malgorzata.borowiec@gmail.com}
\email{mlotkow@math.uni.wroc.pl}
\subjclass[2010]{Primary 05A05; Secondary  20B35}
\keywords{Eulerian numbers, signed permutations, moments of a probability measure.}

\begin{document}

\begin{abstract}
We introduce a new array of type $D$ Eulerian numbers,
different from that studied by Brenti, Chow and Hyatt.
We find in particular the recurrence relation, Worpitzky formula
and the generating function. We also find the probability distributions
whose moments are Eulerian polynomials of type $A$, $B$ and $D$.
\end{abstract}

\maketitle


\section{Introduction}

For a sequence $(a_1,a_2,\ldots,a_s)$, $a_i\in\mathbb{R}$, the \textit{number of descends}
is defined as the cardinality of the set
\[
\big\{i\in\{1,2,\ldots,s-1\}:a_i>a_{i+1}\big\}.
\]

The aim of this paper is to present a new array of type $D$ Eulerian numbers,
different from that studied by Brenti \cite{brenti}, Chow \cite{chow} and Hyatt \cite{hyatt}.
They define $D_{n,k}$ as the number of elements $\sigma$ in the group $\mathcal{D}_n$
(see Section~4 for the definition) such that the sequence
$(-\sigma(2),\sigma(1),\sigma(2),\ldots,\sigma(n))$ has $k$ descends, see entry $A066094$ in OEIS~\cite{oeis}.
For $0\le n\le 4$ these numbers look as follows:
\[
 \begin{array}{cccccccccc}
   		&& & & 1 & & & \\
   	  && & 1 &  & 1 & & \\
      && 1 & & 2 & & 1 & \\
   & 1 & & 11 & & 11 & & 1\\
1 && 44 && 102 && 44 && 1. &
   \end{array}
\]

In this paper we define type $D$ Eulerian numbers $D(n,k)$ by counting descends in the sequence
$(0,\sigma(1),\sigma(2),\ldots,\sigma(n))$, like for type $B$.
First we find the basic recurrence relations,
which involve the numbers $D(n,k)$ together with $\widetilde{D}(n,k):=B(n,k)-D(n,k)$
(Theorem~\ref{ddrpropositionrecurrence}). Then we prove a simple relation between
the numbers $D(n,k)$ and $\widetilde{D}(n,k)$ (Proposition~\ref{ddpropdifferecebinomial})
and between $D(n,k)$ and $B(n,k)$ (Corollary~\ref{ddcorollarydb}).
Next we derive new recurrence relations for $D(n,k)$
and $\widetilde{D}(n,k)$ independently (Proposition~\ref{ddpropositionrecurrenceb}).
We also find Worpitzki type formulas and the generating function
for the numbers $D(n,k)$ and $\widetilde{D}(n,k)$.

We also prove that for $t\ge0$ the sequences
$\left\{P^{\mathrm{A}}_n(t)\right\}_{n=0}^{\infty}$,
$\left\{P^{\mathrm{B}}_n(t)\right\}_{n=0}^{\infty}$ and
$\left\{P^{\mathrm{D}}_n(t)\right\}_{n=0}^{\infty}$
are positive definite, where $P^{\mathrm{A}}_n(t)$,
$P^{\mathrm{B}}_n(t)$ and  $P^{\mathrm{D}}_n(t)$ denote the Eulerian
polynomials of type $A$, $B$ and $D$ respectively.
For $A$ and $B$ it was showed by Barry \cite{barry2011,barry2013}
who computed the Hankel transforms.
Here we indicate the corresponding probability distributions
$\mu^{\mathrm{A}}_{t}$, $\mu^{\mathrm{B}}_{t}$ and $\mu^{\mathrm{D}}_{t}$.

\section{Eulerian numbers of type $A$}

The classical Eulerian numbers
were introduced by Euler as a tool for studying infinite sums of the form (\ref{aaeulersummation}).
We define $A(n,k)$ as the number of permutations $\sigma\in \mathcal{S}_{n}$
having $k$ descends in the sequence $(\sigma(1),\ldots,\sigma(n))$.
Here we record $A(n,k)$ for $0\le n\le4$:
\[
 \begin{array}{cccccccccccc}
   		&& & & & 1 & & & & \\
   	  && & & 1 &  & 0 & & &\\
      && & 1 & & 1 & & 0 & &\\
   && 1 & & 4 & & 1 & & 0&\\
& 1 && 11 && 11 && 1 && 0 & &\\
1 && 26 && 66 && 26 && 1 && 0.
   \end{array}
\]

Recall the main properties of $A(n,k)$,
for details we refer to \cite{stanley,gkp}
and to entry A123125 in OEIS.
The numbers $A(n,k)$ admit the following recurrence relation:
\begin{equation}\label{aarecurrence}
A(n, k)=(n-k)A(n-1,k-1)+(k+1)A(n-1,k)
\end{equation}
for $0<k<n$, with the boundary conditions:
$A(n,0)=1$ for $n\geq0$ and $A(n,n)=0$ for $n\geq1$.
They also can be expressed in the following way:
\begin{equation}\label{aaformula}
A(n,k)
=\sum_{j=0}^{k}(-1)^{k-j} \binom{n+1}{k-j} (j+1)^n.
\end{equation}

Let us also mention \textit{Worpitzky formula}: for $n\ge0$, $x\in\mathbb{R}$ we have
\begin{equation}\label{aaworpitzky}
\sum_{k=0}^{n}\binom{x+k}{n}A(n,k)=x^n.
\end{equation}

The Eulerian polynomials are defined by
\begin{equation}
P^{\mathrm{A}}_{n}(t):=\sum_{k=0}^{n}A(n, k)t^k,
\end{equation}
in particular $P^{\mathrm{A}}_{n}(1)=n!$.
Euler himself encountered $P^{\mathrm{A}}_{n}(t)$
in the formula:
\begin{equation}\label{aaeulersummation}
\sum_{j=1}^{\infty}t^j j^n 
=\frac{t\cdot P^{\mathrm{A}}_{n}(t)}{(1-t)^{n+1}},
\end{equation}
which holds for $n\ge0$, $|t|<1$.
For $n\ge1$ we have $A(n,k)=A(n,n-k-1)$, which implies
\begin{equation}\label{aapolysymetry}
t^{n-1} P^{\mathrm{A}}_{n}(1/t)=P^{\mathrm{A}}_{n}(t).
\end{equation}

The exponential generating function is equal to
\begin{equation}
f^{\mathrm{A}}(t,z):=
\sum_{n=0}^{\infty}\frac{P^{\mathrm{A}}_{n}(t)}{n!}z^n
=\frac{t-1}{t-e^{(t-1)z}}.
\end{equation}

Recall that a sequence $\{a_n\}_{n=0}^{\infty}$, $a_n\in\mathbb{R}$, is
said to be \textit{positive definite} if the infinite matrix
$\left(a_{i+j}\right)_{i,j=0}^{\infty}$ is positive definite,
i.e. for every sequence $\{c_i\}_{i=0}^{\infty}$, $c_i\in\mathbb{R}$,
with only finitely many nonzero entries,
we have
\[
\sum_{i,j=0}^{\infty} a_{i+j}c_i c_j\ge0.
\]
This is equivalent to the fact, that
there exists a nonnegative measure $\mu$ on $\mathbb{R}$
such that $a_n$ are \textit{moments} of $\mu$, i.e. $a_n=\int x^n\,d\mu(x)$,
$n=0,1,\ldots$, see \cite{akhiezer} for details.

It is well known, that $\{n!=P^{\mathrm{A}}_{n}(1)\}_{n=0}^{\infty}$
is positive definite as the moment sequence of the gamma distribution
with weight $e^{-x}$ on $[0,+\infty)$.
Barry \cite{barry2011,barry2013} showed that for $t\ge0$ the sequence
$\left\{P_{n}(t)\right\}_{n=0}^{\infty}$ is positive definite
by computing the Hankel transform and Jacobi parameters.
Here we are going to show, that for $0<t\ne1$,
$P_{n}(t)$ are moments of a dilated geometric distribution.

\begin{theorem}\label{aapropmeasure}
If $t\ge0$ then $\left\{P^{\mathrm{A}}_{n}(t)\right\}_{n=0}^{\infty}$
is the moment sequence of probability measure $\mu^{\mathrm{A}}_{t}$
given by: $\mu^{\mathrm{A}}_{0}=\delta_{1}$,
\begin{align*}
\mu^{\mathrm{A}}_{t}&=\sum_{j=1}^{\infty}(1-t)t^{j-1}\delta_{j(1-t)}\\
\intertext{if $0<t<1$,}
\mu^{\mathrm{A}}_{1}&=
e^{-x}\chi_{[0,+\infty)}(x)\,dx,\\
\intertext{and for $t>1$}
\mu^{\mathrm{A}}_{t}&=\sum_{j=0}^{\infty}\frac{t-1}{t^{j+1}}\delta_{j(t-1)}.
\end{align*}
\end{theorem}

\begin{proof}
We have to show that the $n$-th moment of $\mu^{\mathrm{A}}_{t}$ is $P^{\mathrm{A}}_{n}(t)$.
For $0<t<1$ we have
\[
\int_{\mathbb{R}} t^n\,d\mu^{\mathrm{A}}(t)
=\sum_{j=1}^{\infty}(1-t)^{n+1}j^n t^{j-1}
\]
which, in view of (\ref{aaeulersummation}), is equal to $P^{\mathrm{A}}_{n}(t)$.

Assume that $t>1$. Then $\mu^{\mathrm{A}}_{t}(\mathbb{R})=1$ and for $n\ge1$ we substitute $t:=1/s$:
\[
\int_{\mathbb{R}} t^n\,d\mu^{\mathrm{A}}(t)
=\sum_{j=0}^{\infty}\frac{(t-1)^{n+1}j^n}{t^{j+1}}
=\frac{(1-s)^{n+1}}{s^n}\sum_{j=1}^{\infty}j^n s^{j}=\frac{P^{\mathrm{A}}_{n}(s)}{s^{n-1}}
\]
and by (\ref{aapolysymetry}) this is equal to $P^{\mathrm{A}}_{n}(1/s)=P^{\mathrm{A}}_{n}(t)$.

Finally, by the definition of the gamma function we have
\[
\int_{0}^{+\infty}x^{n}e^{-x}\,dx=\Gamma(n+1)=n!=P^{\mathrm{A}}_{n}(1),
\]
which completes the proof.
\end{proof}

\section{Eulerian numbers of type $B$}

We define $\mathcal{B}_n$ as the group of such permutations $\sigma$ of the set
\[
\{-n,\ldots,-1,0,1,\ldots,n\}
\]
that $\sigma$ is odd, i.e.
$\sigma(-k)=-\sigma(k)$ for every $k$,
$-n\le k\le n$.
Then $\left|\mathcal{B}_n\right|=2^n n!$ and $\mathcal{B}_n$ can be naturally viewed as a subgroup
of the symmetric group $\mathcal{S}_{2n}$.
For $\sigma\in\mathcal{B}_{n}$ we denote by $\mathrm{desc}(\sigma)$
the number of descends in the sequence $(0,\sigma(1),\ldots,\sigma(n))$.

Denote by $\mathcal{B}_{n,k}$ the set
$\left\{\sigma\in\mathcal{B}_n:\mathrm{desc}(\sigma)=k\right\}$,
and by $B(n,k)$ its cardinality, $0\le k\le n$.
$B(n,k)$ are called \textit{type $B$ Eulerian numbers}.
Here we record $B(n,k)$ for $0\le n\le 4$:
\[
 \begin{array}{cccccccccc}
   		&& & & 1 & & & \\
   	  && & 1 &  & 1 & & \\
      && 1 & & 6 & & 1 & \\
   & 1 & & 23 & & 23 & & 1\\
1 && 76 && 230 && 76 && 1. & \\
   \end{array}
\]
Now we recall basic properties of $B(n,k)$,
for details we refer to \cite{brenti,chowgessel2007}
and to entry $A060187$ in OEIS.
First of all, they satisfy the recurrence relation:
\begin{equation}\label{bbrecurrence}
B(n,k)=(2n-2k+1)B(n-1,k-1)+(2k+1)B(n-1,k),
\end{equation}
$0<k<n$, with the boundary conditions $B(n,0)=B(n,n)=1$.
Similarly to (\ref{aaformula}) and (\ref{aaworpitzky}) we have equality
\begin{equation}\label{bbformula}
B(n,k)
=\sum_{j=0}^{k}(-1)^{k-j}\binom{n+1}{k-j}(2j+1)^{n},
\end{equation}
$0\le k\le n$, and Worpitzky formula of the form:
\begin{equation}\label{bbworpitzky}
\sum_{k=0}^{n}\binom{x+k}{n}B(n,k)=(1+2x)^n
\end{equation}
for $n\ge0$.

The Eulerian polynomials of type $B$,
\begin{equation}
P^{\mathrm{B}}_{n}(t):=\sum_{k=0}^{n}B(n, k)t^k,
\end{equation}
are related to $P^{\mathrm{A}}_{n}(t)$ in the following way:
\begin{equation}
(1+t)^{n+1} P^{\mathrm{A}}_{n}(t)-2^{n} t P^{\mathrm{A}}_{n}(t^2)=P^{\mathrm{B}}_{n}(t^2).
\end{equation}
The symmetry $B(n,k)=B(n,n-k)$ implies $t^n\cdot P^{\mathrm{B}}_{n}(1/t)=P^{\mathrm{B}}_{n}(t)$.
For $|t|<1$, $n\ge0$, we have odd version of (\ref{aaeulersummation}):
\begin{equation}\label{bbeulersummation}
\sum_{k=0}^{\infty}(2k+1)^n t^k=\frac{P^{\mathrm{B}}_{n}(t)}{(1-t)^{n+1}}.
\end{equation}

The exponential generating function is equal
\begin{equation}\label{bbgeneratingfunction}
f^{\mathrm{B}}(t,z)
:=\sum_{n=0}^{\infty}\frac{P^{\mathrm{B}}_n(t)}{n!}z^n
=\frac{(1-t)e^{(1-t)z}}{1-t e^{2(1-t)z}}.
\end{equation}

Similarly as for type $A$ we can prove that the sequence $\left\{P^{\mathrm{B}}_{n}(t)\right\}_{n=0}^{\infty}$
is positive definite for $t\ge0$ (c.f. \cite{barry2011,barry2013}) and indicate the corresponding probability measures.

\begin{theorem}\label{bbpropositionmeasure}
If $t\ge0$ then $\left\{P^{\mathrm{B}}_{n}(t)\right\}_{n=0}^{\infty}$
is the moment sequence of probability measure $\mu^{\mathrm{B}}_{t}$
given by: $\mu^{\mathrm{B}}_{0}=\delta_{1}$,
\begin{align*}
\mu^{\mathrm{B}}_{t}&=\sum_{k=0}^{\infty}(1-t)t^k\delta_{(2k+1)(1-t)}\\
\intertext{if $0<t<1$,}
\mu^{\mathrm{B}}_{1}&=
\frac{1}{2}e^{-x/2}\chi_{[0,+\infty)}(x)\,dx,\\
\intertext{and for $t>1$}
\mu^{\mathrm{B}}_{t}&=\sum_{k=0}^{\infty}\frac{t-1}{t^{k+1}}\delta_{(2k+1)(t-1)}.
\end{align*}
\end{theorem}

\begin{proof}
The case $0<t<1$ is a consequence of (\ref{bbeulersummation}).
For $t>1$ we substitute $t:=1/s$ and have
\[
\int_{\mathbb{R}} t^n\,d\mu^{\mathrm{B}}(t)
=\sum_{k=0}^{\infty}\frac{(t-1)^{n+1}(2k+1)^n}{t^{k+1}}
\]
\[
=\frac{(1-s)^{n+1}}{s^n}\sum_{j=0}^{\infty}(2k+1)^n s^{k}=\frac{P^{\mathrm{B}}_{n}(s)}{s^{n}}
=P^{\mathrm{B}}_{n}(1/s)=P^{\mathrm{B}}_{n}(t).
\]
Finally
\[
\frac{1}{2}\int_{0}^{+\infty}x^{n}e^{-x/2}\,dx=2^n n!=P^{\mathrm{B}}_{n}(1).
\]
\end{proof}

\section{Eulerian numbers of type $D$}

By $\mathcal{D}_n$ we will denote the group of such elements $\sigma\in\mathcal{B}_n$
that the set $\{\sigma(1),\ldots,\sigma(n)\}$ contains even
number of negative terms.
If $n\ge1$ then $\mathcal{D}_n$ is a normal subgroup
of $\mathcal{B}_n$ of index~2. Denote $\widetilde{\mathcal{D}}_n:=\mathcal{B}_n\setminus\mathcal{D}_n$ and
\[
\mathcal{D}_{n,k}:=\left\{\sigma\in\mathcal{D}_n:\mathrm{desc}(\sigma)=k\right\},
\]
\[
\widetilde{\mathcal{D}}_{n,k}
:=\left\{\sigma\in\widetilde{\mathcal{D}}_n:\mathrm{desc}(\sigma)=k\right\},
\]
so that $\mathcal{D}_{n,k}=\mathcal{B}_{n,k}\cap\mathcal{D}_{n}$,
$\widetilde{\mathcal{D}}_{n,k}=\mathcal{B}_{n,k}\setminus\mathcal{D}_{n}$.
Cardinalities of these sets will be denoted $D(n,k)$ and $\widetilde{D}(n,k)$
respectively.
Since $\mathcal{B}_{n,k}=\mathcal{D}_{n,k}\dot{\cup}\widetilde{\mathcal{D}}_{n,k}$,
we have
\begin{equation}\label{ddplusddisbb}
B(n,k)=D(n,k)+\widetilde{D}(n,k).
\end{equation}
Here we record the numbers $D(n,k)$ (the \textit{primary type $D$ triangle}):
\begin{equation}
\nonumber
 \begin{array}{cccccccccc}
   		&& & & 1 & & & \\
   	  && & 1 &  & 0 & & \\
      && 1 & & 2 & & 1 & \\
   & 1 & & 10 & & 13 & & 0\\
1 && 36 && 118 && 36 && 1 &
   \end{array}
\end{equation}
and $\widetilde{D}(n,k)$ (the \textit{complementary type $D$ triangle}):
\begin{equation}
\nonumber
 \begin{array}{cccccccccc}
   		&& & & 0 & & & \\
   	  && & 0 &  & 1 & & \\
      && 0 & & 4 & & 0 & \\
   & 0 & & 13 & & 10 & & 1\\
0 && 40 && 112 && 40 && 0 &
   \end{array}
\end{equation}
 for $0\le n\le 4$.
We conjecture that every row in $D$ and in $\widetilde{D}$ constitutes a unimodal sequence
and that for every $k\ge1$ the sequences
$\{D(n+k,k)\}_{n=0}^{\infty}$, $\{D(n+k,n)\}_{n=0}^{\infty}$,
$\{\widetilde{D}(n+k,k)\}_{n=0}^{\infty}$,  $\{\widetilde{D}(n+k,n)\}_{n=0}^{\infty}$
are increasing.

First we note the following symmetry.

\begin{proposition}
For $0\le k\le n$ we have
\[
D(n,k)=D(n,n-k),\quad
\widetilde{D}(n,k)=\widetilde{D}(n,n-k)
\]
if $n$ is even and
\[
D(n,k)=\widetilde{D}(n,n-k),\quad
\widetilde{D}(n,k)=D(n,n-k)
\]
if $n$ is odd.
\end{proposition}

\begin{proof}
For $\sigma\in\mathcal{B}_n$ define $-\sigma\in\mathcal{B}_n$
by $(-\sigma)(k):=-\sigma(k)$.
It suffices to note that the map $\sigma\mapsto-\sigma$
defines bijections $\mathcal{D}_{n,k}\to\mathcal{D}_{n,n-k}$,
$\widetilde{\mathcal{D}}_{n,k}\to\widetilde{\mathcal{D}}_{n,n-k}$
if $n$ is even and $\mathcal{D}_{n,k}\to\widetilde{\mathcal{D}}_{n,n-k}$,
$\widetilde{\mathcal{D}}_{n,k}\to\mathcal{D}_{n,n-k}$
if $n$ is odd.
\end{proof}

Now we present the fundamental recurrence relations for both triangles.

\begin{theorem}\label{ddrpropositionrecurrence}
The numbers $D(n,k)$, $\widetilde{D}(n,k)$ admit the following recurrence relations:
\[
D(n, k)=(k+1)D(n-1, k)+(n-k)D(n-1, k-1)
\]
\[
+k\widetilde{D}(n-1, k)+(n-k+1)\widetilde{D}(n-1, k-1)
\]
and
\[
\widetilde{D}(n, k)=(k+1)\widetilde{D}(n-1, k)+(n-k)\widetilde{D}(n-1, k-1)
\]
\[
+kD(n-1, k)+(n-k+1)D(n-1, k-1)
\]
for $0<k<n$ and the boundary conditions: $D(n, 0)=1$, $\widetilde{D}(n, 0)=0$ and
\[
D(n,n)=\left\{\begin{array}{ll}
1&\mbox{if $n$ is even,}\\
0&\mbox{if $n$ is odd,}
\end{array}\right.
\qquad
\widetilde{D}(n,n)=\left\{\begin{array}{ll}
0&\mbox{if $n$ is even,}\\
1&\mbox{if $n$ is odd.}
\end{array}\right.
\]
\end{theorem}

\begin{proof}
For the sake of this proof we will identify an element $\sigma\in\mathcal{B}_{n}$
with the sequence $(\sigma_0,\ldots,\sigma_n)$,
where we write $\sigma_k$ instead of $\sigma(k)$.
Now for $(\sigma_0,\ldots,\sigma_n)\in\mathcal{B}_{n}$ we define
\[
\Lambda\sigma:=(\sigma_0,\ldots,\widehat{\sigma}_j,\ldots,\sigma_{n})\in\mathcal{B}_{n-1},
\]
where $j$ is such that $\sigma_j=\pm n$, and $\widehat{\sigma}_j$ means, that
the element ${\sigma}_j$ has been removed from the sequence.

For given $\sigma\in\mathcal{D}_{n,k}$, $0<k<n$, we have four possibilities:
\begin{itemize}
\item
$\sigma_j=n$ and either $j=n$ or $\sigma_{j-1}>\sigma_{j+1}$, $1\le j<n$.
Then $\Lambda\sigma\in\mathcal{D}_{n-1,k}$.

\item $\sigma_j=n$ and $\sigma_{j-1}<\sigma_{j+1}$, $1\le j<n$.
Then $\Lambda\sigma\in\mathcal{D}_{n-1,k-1}$.

\item $\sigma_j=-n$ and $\sigma_{j-1}>\sigma_{j+1}$, $1\le j<n$.
Then $\Lambda\sigma\in\widetilde{\mathcal{D}}_{n-1,k}$.

\item $\sigma_j=-n$ and either $j=n$ or $\sigma_{j-1}<\sigma_{j+1}$, $1\le j<n$.
Then $\Lambda\sigma\in\widetilde{\mathcal{D}}_{n-1,k-1}$.
\end{itemize}

Now, suppose we are given a fixed $\tau=(\tau_0,\ldots,\tau_{n-1})$
which belongs to one of the sets $\mathcal{D}_{n-1,k}$, $\mathcal{D}_{n-1,k-1}$,
$\widetilde{\mathcal{D}}_{n-1,k}$ or $\widetilde{\mathcal{D}}_{n-1,k-1}$.
We are looking for $\sigma\in\mathcal{D}_{n,k}$ such that $\Lambda\sigma=\tau$.

If $\tau\in\mathcal{D}_{n-1,k}$ then
we should either put $n$ to the end of $\tau$, or insert into a drop of $\tau$,
i.e. between $\tau_{i-1}$ and $\tau_{i}$, where $1\le i\le n-1$, $\tau_{i-1}>\tau_{i}$,
so we have $k+1$ possibilities.

Similarly, if $\tau\in\mathcal{D}_{n-1,k-1}$
then we construct $\sigma$ by inserting $n$ between $\tau_{i-1}$ and $\tau_{i}$, $1\le i\le n-1$,
where $\tau_{i-1}<\tau_{i}$. For this we have $n-k$ possibilities.

Now assume that $\tau\in\widetilde{\mathcal{D}}_{n-1,k}$.
Now we should insert $-n$ between
$\tau_{i-1}$ and $\tau_{i}$, $1\le i\le n-1$, where $\tau_{i-1}>\tau_{i}$,
for which we have $k$ possibilities.

Finally, if  $\tau\in\widetilde{\mathcal{D}}_{n-1,k-1}$
then we put $-n$ either at the end of $\tau$ or between $\tau_{i-1}$
and $\tau_{i}$, $1\le i\le n-1$, where $\tau_{i-1}<\tau_{i}$, for which we have $n-k+1$ possibilities.

Therefore the number of $\sigma\in\mathcal{D}_{n,k}$
such that $\Lambda\sigma$ belongs to the set $\mathcal{D}_{n-1,k}$,
$\mathcal{D}_{n-1,k-1}$, $\widetilde{\mathcal{D}}_{n-1,k}$ or
$\widetilde{\mathcal{D}}_{n-1,k-1}$
is equal to $(k+1)D(n-1, k)$, $(n-k)D(n-1, k-1)$,
$k\widetilde{D}(n-1, k)$ or $(n-k+1)\widetilde{D}(n-1, k-1)$
respectively. This proves the recurrence for the numbers $D(n,k)$.
For $\widetilde{D}(n,k)$ we proceed in the same way.

The boundary conditions are quite obvious:
$(0,1,2,\ldots,n)$ is the only element $\sigma\in\mathcal{B}_{n}$
such that $\mathrm{desc}(\sigma)=0$ and $(0,-1,-2,\ldots,-n)$
is the only element $\tau\in\mathcal{B}_{n}$ with
$\mathrm{desc}(\tau)=n$. The former belongs to $\mathcal{D}_{n}$,
while the latter belongs to $\mathcal{D}_{n}$ if and only if $n$ is even.
\end{proof}

Now we show a simple relation between the numbers
$D(n,k)$ and $\widetilde{D}(n,k)$.

\begin{proposition}\label{ddpropdifferecebinomial}
For $0\le k\le n$ we have
\begin{equation}\label{dddifferencebinomial}
D(n,k)-\widetilde{D}(n,k)=(-1)^k\binom{n}{k}.
\end{equation}
\end{proposition}

\begin{proof}
For $k=0$ and $k=n$ this is straightforward. Assume that $0<k<n$ and that the equality holds for $n-1$.
From Proposition~\ref{ddrpropositionrecurrence} we have
\[
D(n,k)-\widetilde{D}(n,k)=
(k+1)\binom{n-1}{k}(-1)^k+(n-k)\binom{n-1}{k-1}(-1)^{k-1}
\]
\[-k\binom{n-1}{k}(-1)^k-(n-k+1)\binom{n-1}{k-1}(-1)^{k-1}
\]
\[
=\binom{n-1}{k}(-1)^k-\binom{n-1}{k-1}(-1)^{k-1}=\binom{n}{k}(-1)^k.
\]
\end{proof}

Combining (\ref{ddplusddisbb}) with (\ref{dddifferencebinomial}) we get

\begin{corollary}\label{ddcorollarydb}
\begin{align}
D(n,k)&=\frac{1}{2}B(n,k)+\frac{1}{2}\binom{n}{k}(-1)^k,\\
\widetilde{D}(n,k)&=\frac{1}{2}B(n,k)-\frac{1}{2}\binom{n}{k}(-1)^k.
\end{align}
\end{corollary}

Now we are able to provide independent recurrences for $D(n,k)$ and $\widetilde{D}(n,k)$.

\begin{proposition}\label{ddpropositionrecurrenceb}
The numbers $D(n,k)$, $\widetilde{D}(n,k)$ admit the following recurrence:
\begin{align*}
D(n, 0)&=1,& \widetilde{D}(n, 0)&=0,\\
D(n,n)&=\frac{1 + (-1)^n}{2},&\widetilde{D}(n, n)&=\frac{1 - (-1)^n}{2}
\end{align*}
and for $0<k<n$
\begin{align*}
D(n, k)&=(2 k + 1)D(n - 1, k) + (2 n - 2 k + 1)D(n - 1,k - 1) + \binom{n - 1}{k - 1}(-1)^k,\\
\widetilde{D}(n, k)&= (2 k + 1) \widetilde{D}(n - 1, k) + (2 n - 2 k + 1)\widetilde{D}(n-1,k-1) -\binom{n - 1}{k - 1} (-1)^k.
\end{align*}
\end{proposition}

\begin{proof}
This is a consequence of Proposition~\ref{ddrpropositionrecurrence} and (\ref{dddifferencebinomial}).
\end{proof}

Now we can prove the following
Worpitzky type identities.

\begin{proposition}
For $n\ge0$, $x\in\mathbb{R}$ we have:
\begin{align}
2\sum_{k=0}^{n}\binom{x+k}{n}D(n,k)&=(2x+1)^n+(-1)^n,\\
2\sum_{k=0}^{n}\binom{x+k}{n}\widetilde{D}(n,k)&=(2x+1)^n-(-1)^n.
\end{align}
\end{proposition}

\begin{proof}
These formulas follow from Corollary~\ref{ddcorollarydb}, (\ref{bbworpitzky})
and from the identity
\begin{equation}
\sum_{k=0}^{n}\binom{x+k}{n}\binom{n}{k}(-1)^k=(-1)^n,
\end{equation}
see for example (5.24) in \cite{gkp}.
\end{proof}

Define polynomials
\begin{align*}
P^{\mathrm{D}}_n(t)&:=\sum_{k=0}^n D(n,k)t^k,\\
P^{\widetilde{\mathrm{D}}}_n(t)&:=\sum_{k=0}^n\widetilde{D}(n,k)t^k.
\end{align*}
From (\ref{ddplusddisbb}) and (\ref{dddifferencebinomial}) we have
\begin{align}
P^{\mathrm{D}}_n(t)+P^{\widetilde{\mathrm{D}}}_n(t)&=P^{\mathrm{B}}_n(t),\\
P^{\mathrm{D}}_n(t)-P^{\widetilde{\mathrm{D}}}_n(t)&=(1-t)^n,
\end{align}
which implies
\begin{align}
P^{\mathrm{D}}_n(t)&=\frac{1}{2}P^{\mathrm{B}}_n(t)+\frac{1}{2}(1-t)^n,\label{ddgenbb0}\\
P^{\widetilde{\mathrm{D}}}_n(t)&=\frac{1}{2}P^{\mathrm{B}}_n(t)-\frac{1}{2}(1-t)^n.\label{ddgenbb1}
\end{align}

\begin{proposition}
For $n\ge1$ we have
\begin{align}
P^{\mathrm{D}}_n(t)&=(2nt-t+1)P^{\mathrm{D}}_{n-1}(t)
+2t(1-t)\frac{\mathrm{d}}{\mathrm{d}t}P^{\mathrm{D}}_{n-1}(t)-t(1-t)^{n-1},\label{ddpoly0}\\
P^{\widetilde{\mathrm{D}}}_n(t)&=(2nt-t+1)P^{\widetilde{\mathrm{D}}}_{n-1}(t)
+2t(1-t)\frac{\mathrm{d}}{\mathrm{d}t}P^{\widetilde{\mathrm{D}}}_{n-1}(t)+t(1-t)^{n-1}.
\end{align}
\end{proposition}

\begin{proof}
We have
\[
\sum_{k=0}^{n}(2k+1)D(n-1,k)t^k=
2t\frac{\mathrm{d}}{\mathrm{d}t}P^{\mathrm{D}}_{n-1}(t)+P^{\mathrm{D}}_{n-1}(t),
\]
\[
\sum_{k=0}^{n}(2n-2k+1)D(n-1,k-1)t^k=
(2n-1)t P^{\mathrm{D}}_{n-1}(t)-2t^2\frac{\mathrm{d}}{\mathrm{d}t}P^{\mathrm{D}}_{n-1}(t)
\]
and
\[
\sum_{k=0}^{n}\binom{n-1}{k-1}(-1)^k t^k=-t(1-t)^{n-1}.
\]
Summing up and applying Proposition~\ref{ddpropositionrecurrenceb} we get the first equality.
The second one we obtain in the same way.
\end{proof}

Now we are able to find the generating functions.

\begin{proposition}
The exponential generating functions
\[
f^{\mathrm{D}}(t,z):=\sum_{n=0}^{\infty}\frac{P^{\mathrm{D}}_n(t)}{n!}z^n,
\]
\[
f^{\widetilde{\mathrm{D}}}(t,z):=\sum_{n=0}^{\infty}\frac{P^{\widetilde{\mathrm{D}}}_n(t)}{n!}z^n
\]
satisfy the following differential equations
\[
(1+t)f^{\mathrm{D}}(t,z)+(2tz-1)\frac{\partial f^{\mathrm{D}}}{\partial z}(t,z)
+2t(1-t)\frac{\partial f^{\mathrm{D}}}{\partial t}(t,z)=t e^{(1-t)z},
\]
\[
(1+t)f^{\widetilde{\mathrm{D}}}(t,z)
+(2tz-1)\frac{\partial f^{\widetilde{\mathrm{D}}}}{\partial z}(t,z)
+2t(1-t)\frac{\partial f^{\widetilde{\mathrm{D}}}}{\partial t}(t,z)=-t e^{(1-t)z},
\]
with the initial conditions $f^{\widetilde{\mathrm{D}}}(t,0)=1$,
$f^{\widetilde{\mathrm{D}}}(t,0)=0$. We have
\begin{align}
f^{\mathrm{D}}(t,z)&=\frac{(2-t)e^{(1-t)z}-t e^{3(1-t)z}}{2-2t e^{2(1-t)z}},\label{ddgeneratingfunction0}\\
f^{\widetilde{\mathrm{D}}}(t,z)&=\frac{t e^{3(1-t)z}-t e^{(1-t)z}}{2-2t e^{2(1-t)z}}.\label{ddgeneratingfunction1}
\end{align}
\end{proposition}

\begin{proof}
Multiplying both sides of (\ref{ddpoly0}) by $z^n/n!$ and taking sum $\sum_{n=0}^{\infty}$
we obtain equation
\begin{equation}\label{dddiffeqproof}
f^{\mathrm{D}}=2tz f^{\mathrm{D}}+(1-t)F^{\mathrm{D}}
+2t(1-t)\frac{\partial F^{\mathrm{D}}}{\partial t}-\frac{t}{1-t}e^{(1-t)z},
\end{equation}
where
\[
F^{\mathrm{D}}(t,z):=\sum_{n=0}^{\infty}\frac{P^{\mathrm{D}}_{n-1}(t)}{n!}z^n.
\]
Taking the derivative $\frac{\partial}{\partial z}$
of both sides of (\ref{dddiffeqproof}) and keeping in mind that
$\frac{\partial F^{\mathrm{D}}}{\partial z}=f^{\mathrm{D}}$, we get the first
differential equation. The second is obtained in the same way.

Formulas (\ref{ddgeneratingfunction0},\ref{ddgeneratingfunction1})
follow directly from (\ref{bbgeneratingfunction}) and (\ref{ddgenbb0},\ref{ddgenbb1}).
\end{proof}

Finally we prove positive definiteness of the sequence $\left\{P^{\mathrm{D}}_{n}(t)\right\}_{n=0}^{\infty}$.

\begin{theorem}
If $t\ge0$ then $\left\{P^{\mathrm{D}}_{n}(t)\right\}_{n=0}^{\infty}$
is the moment sequence of probability measure $\mu^{\mathrm{D}}_{t}$
given by: $\mu^{\mathrm{D}}_{0}=\delta_{1}$,
\begin{align*}
\mu^{\mathrm{D}}_{t}&=\frac{2-t}{2}\delta_{1-t}+\sum_{k=1}^{\infty}\frac{(1-t)t^k}{2}\delta_{(2k+1)(1-t)}\\
\intertext{if $0<t<1$,}
\mu^{\mathrm{D}}_{1}&=
\frac{1}{2}\delta_{0}+\frac{1}{4}e^{-x/2}\chi_{[0,+\infty)}(x)\,dx,\\
\intertext{and for $t>1$}
\mu^{\mathrm{D}}_{t}&=\frac{1}{2}\delta_{1-t}+\sum_{k=0}^{\infty}\frac{t-1}{2t^{k+1}}\delta_{(2k+1)(t-1)}.
\end{align*}
\end{theorem}

\begin{proof}
From (\ref{ddgenbb0})
we have $\mu^{\mathrm{D}}_{t}=\frac{1}{2}\mu^{\mathrm{B}}_{t}+\frac{1}{2}\delta_{1-t}$
and now we apply Theorem~\ref{bbpropositionmeasure}.
\end{proof}

Since $\left\{P^{\widetilde{\mathrm{D}}}_n(t)\right\}_{n=0}^{\infty}$
is the moment sequence of the non-positive measure $\frac{1}{2}\mu^{\mathrm{B}}_{t}-\frac{1}{2}\delta_{1-t}$,
it is not positive definite. It turns out however,
that for $t\ge1$ the sequence $\left\{P^{\widetilde{\mathrm{D}}}_{n+1}(t)\right\}_{n=0}^{\infty}$
is positive definite.

\begin{proposition}
For $t\ge1$ the sequence $\left\{P^{\widetilde{\mathrm{D}}}_{n+1}(t)/t\right\}_{n=0}^{\infty}$
is positive definite and the corresponding probability measure $\nu_{t}$ is
\begin{align*}
\nu_1&=\frac{1}{4}xe^{-x/2}\chi_{[0,+\infty)}(x)\,dx\\
\intertext{and for $t>1$}
\nu_{t}&=\frac{t-1}{2t}\delta_{1-t}
+\sum_{k=0}^{\infty}\frac{(t-1)^2(2k+1)}{2t^{k+2}}\delta_{(t-1)(2k+1)}.
\end{align*}
\end{proposition}

\begin{proof}
We have
\[
P^{\widetilde{\mathrm{D}}}_{n+1}(t)=\int x^n\cdot x\,d \left(\frac{1}{2}\mu^{\mathrm{B}}_{t}-\frac{1}{2}\delta_{1-t}\right)(x)
\]
and the measure
\[
\nu_{t}=\frac{x}{t}\,d \left(\frac{1}{2}\mu^{\mathrm{B}}_{t}-\frac{1}{2}\delta_{1-t}\right)(x)
\]
is positive for $t\ge1$.
\end{proof}

\end{document}